\author{A. Sedunova}
\documentclass[english]{article}
\usepackage[T1]{fontenc}
\usepackage[latin9]{inputenc}
\usepackage{endnotes}
\usepackage{amsmath}
\usepackage{amssymb}
\usepackage{comment}

\makeatletter
\numberwithin{figure}{section}
\numberwithin{table}{section}
\numberwithin{equation}{section}
 \let\footnote=\endnote

\newenvironment{proof}{{\bf Proof.}}{$\hfill \square $}
\newtheorem{example}{{\bf Example}}{}

\DeclareMathOperator{\mmod}{mod}
\DeclareMathOperator{\ord}{ord}
\DeclareMathOperator{\prob}{Prob}

\newcommand*{\bfrac}[2]{\genfrac{}{}{0pt}{}{#1}{#2}}

\usepackage[a4paper, top=30mm, left=25mm, right=25mm, bottom=30mm]{geometry}

\newtheorem{lemma}{Lemma}
\newtheorem{theorem}{Theorem}
\newtheorem{corollary} {Corollary}

{}

\makeatother

\usepackage{babel}
\begin{document}
\global\long\def\veps{\varepsilon}
\global\long\def\k{\kappa}
\global\long\def\s{\sigma}
\global\long\def\om{\omega}
\global\long\def\r{\rho}
\global\long\def\P{\mathbf{P}}

\global\long\def\R{\mathbb{R}}
\global\long\def\F{\mathbb{F}}
\global\long\def\E{\mathbb{E}}
\global\long\def\W{\mathcal{W}}
\global\long\def\Pc{\mathcal{P}}
\global\long\def\Cc{\mathcal{C}}

\global\long\def\le{\leqslant}
\global\long\def\ge{\geqslant}

\title{On the Bombieri-Pila Method Over Function Fields}
\maketitle
\begin{abstract}
In \cite{b-pila} E. Bombieri and J. Pila introduced a method for bounding
the number of integral lattice points that belong to a given arc
under several assumptions. 
In this paper we generalize the Bombieri-Pila method to
the case of function fields of genus 0 in one variable. We then apply the result to counting the number of elliptic curves contain in an isomorphism class and with coefficients in a box.
\end{abstract}

\section{Introduction}

In  \cite{b-pila}  E. Bombieri and J. Pila proved that
if $\Gamma$ is a subset of an irreducible algebraic curve of degree $d$ inside a square of side $N$, then the number of lattice points on $\Gamma$ is bounded by $c(d,\veps)N^{\frac{1}{d}+\veps}$ for
any $\veps>0$, where the constant $c(d,\veps)$ does not depend on
$\Gamma$. There are many analogues of this remarkable
result. For example, one can be interested in finding a bound for
a number of solutions of $f(x,y)=0 \mod p$ with $x\in I$,
$y\in J$, where $I$ and $J$ are short intervals in $\mathbb{Z}/p\mathbb{Z}$ (see \cite{shparlinski1} and \cite{shparlinski2}).  Such results are $p$-analogues
of the Bombieri-Pila bound. (Here we should assume that the lengths
of $I$ and $J$ are much shorter than $p$, so that the Weil bound and
other standard methods cannot be applied.)

One can go further and look for a function field analogue. Here we work in a finite field $\F_{q^n}$ modelled as  $\F_q[T]/f(T)$ where $f$ is
a fixed irreducible polynomial of degree $n$ and $T$ is a formal variable. Then an
interval is the set of polynomials of the form $X+Y=X(T)+Y(T)$, where $X\in \F_q[T]$ is a fixed polynomial
and $Y(T)$ runs through all polynomials of degree bounded by a given natural number. 
This point of view was used by J. Cilleruelo and I. Shparlinski in \cite{shparlinski3} for obtaining some bounds on the number of solutions of polynomial congruences modulo a prime with variables in short intervals. The same authors also formulated \cite[Problem 9]{shparlinski3}, which is solved here.

Our main goal is to prove

\begin{theorem}\label{thm}

Let $\Cc$ be an irreducible algebraic curve of degree $d$ over $\F_q[T]$, $q$ is a prime power. Define $S$ as the set of points on $\Cc$ inside $I^2$, where $I$ is a set of polynomials $X \in \F_q[T]$ with $\deg X \le n$ and $|I|=q^{n+1}$. 
Then
$$|S|\ll_{d,\veps} |I|^{\frac{1}{d}+\veps}.$$

\end{theorem}

One can pose a question: why can we not just follow the Bombieri-Pila approach in order to get Theorem \ref{thm}?
Unfortunately, in this case we will cross some difficulties in getting Lemma 2 of \cite{b-pila}, since we do not have the necessary analogue of the mean value theorem in function fields (see \cite{sw-dy}, Lemma 1). There seem to be at least two plausible ways to avoid this difficulty. The first one consists in getting a function field variant of Theorem 4  in Heath-Brown's article \cite{h-b}. The second one, which we will follow here, is to adapt the method of Helfgott-Venkatesh \cite{helf-ven}. 

We will need analogues of Propositions 3.1 and 3.2 of \cite{helf-ven}. Combining and developing the original ideas of \cite{b-pila} together with an adaptation of some results of \cite{helf-ven} will lead us to our main result.

After that we will use Theorem \ref{thm} to get some applications, such as a calculation of the number of isomorphism classes which are represented by elliptic curves $E_{a,b}$ parametrized by coefficients $a,b \in \F_q[T]$ lying in a small box, say, $I^2$. Using this result one can calculate the number of elliptic curves lying in a given isomorphism class with coefficients lying in a small box. To proceed we will work with ideas proposed in \cite{shparlinski2}.

\section{Auxiliary statements}

Let $X$ and $Y$ be
variables with values in $\F_{q}[T]$, i.e. their values are of the form $X=X(T)=a_{0}+a_{1}T+\ldots+a_{n}T^{n}$, $Y=Y(T)=b_{0}+b_{1}T+\ldots+b_{m}T^{m}$,
where $T$ is a place holder, $a_{i},b_{j}\in\F_{q}$, $i=0,\ldots, \deg X=n$, $j = 0,\ldots,\deg Y=m$. For $X\in \F_q[T]$ we denote by $|X|$ its norm: $|X|=q^{\deg X}$.

Define "an interval" $I$  as the set of polynomials on a formal variable
$T$ of the form $X(T)+Y(T)$, where $X(T)$ is a fixed polynomial
and $Y(T)$ runs through all polynomials of degree less or equal than a given integer.

In what follows $\Cc$ is an irreducible algebraic curve of degree $d$ over $\F_q[T]$, which is described by $F(X,Y)=0$, $F(X,Y) \in (\F_q[T])[X,Y]$. Write $S$ for the set of points on $\Cc$ inside $I^2$.

For any $F(X,Y) \in (\F_q[T])[X,Y]$ we write $\deg_X F$ and $\deg_T F$ to denote the degree of a polynomial $F$ with respect to $X$ and $T$ respectively. 
We also use the standard notation $\deg F(X,Y)$ for
the degree of $F(X,Y)$ as a polynomial in $X$ and $Y$.

Let $\W$ be a set consisting of finitely many linearly independent polynomials $F \in (\F_q[T])[X,Y]$ including the constant polynomial ${\bf 1}$. Write $d_\W$ for the total degree of all elements of $\W$.
Assume that the elements of $\W$ separate points, meaning that $\forall (X_1,Y_1),(X_2,Y_2) \in (\F_q[T])^2$ there is an $F \in \W$ such that $F(X_1,Y_1) \neq F(X_2,Y_2)$. 
We define a $\W$-curve to be an affine algebraic curve described by an equation $G(X,Y)=0$, where all the monomials of $G$ belong to $\W$.\\
During the proof of Theorem \ref{thm} we will use the following choice of $\W$:\newline
\begin{example}\label{wset}
Define $\W=\W_{d,M}$ as
$$\W=\{X^iY^j|\;i \le d, j\le M\},$$
where $d$ and $M$ are given numbers.
Then $|\W|=(d+1)(M+1)$,
$d_{\W}=(d+1)(M+1)\frac{d+M}{2}$.
The $\W$-curves are plane curves of degree less or equal  than $d$ and $M$ in $X$ and $Y$ respectively.
\end{example}
This choice is taken straight from the work of Bombieri and Pila \cite{b-pila}.

\begin{lemma}\label{helf-ven}
Let $\Cc$ be an irreducible algebraic curve of degree $d$ over $\F_q[T]$ and let $S$ be the set of points on $\Cc$ inside $I^2$.
Suppose that 
the number of residues $\{(X,Y) \mmod f,\; X,Y \in S\}$ is at most $\alpha |f|$ for some fixed $\alpha >0$ and for every irreducible polynomial $f \in \F_q[T]$. 
Assume that $\W$ is chosen in a way that any $\W$-curve contains at most constant number $C$ of elements of $S$.
Then the following holds
$$|S|\ll_{\W} |I|^{\frac{2\alpha d_\W}{\om(\om-1)}+o_{\alpha,C}(1)},$$
where $\om=|\W|$. 
\end{lemma}

\begin{proof}
We are going to prove it in the spirit of \cite[Proposition 3.1]{helf-ven}.
Write $P=(X,Y)$ for a point in $(\F_q[T])^2$ with coordinates $X,Y \in \F_q[T]$.
Fixing an arbitrary ordering $F_1,F_2,\ldots,F_\om$ for the elements of $\W$, we define a function
$$W:((\F_q[T])^2)^\om \to \F_q[T]$$
by 
$$W(P_1,\ldots,P_\om)=\det(F_i(P_j))_{1\le i,j\le\om}.$$
Let $\P$ denote an ensemble of points in $S$: $\P=(P_1,\ldots,P_\om)$, $P_i=(X_i,Y_i) \in S$. We say that $\P$ is admissible if $W(\P) = W(P_1,\ldots,P_\om)\neq {\bf 0}$ (where ${\bf 0}$ stands for zero polynomial in $\F_q[T]$). Define
$$\Delta=\prod_{\P}^*|W(\P)|,$$
where $*$ means that we take the operation over all admissible $\P$.\\
By the definition of $d_\W$ we have 
$$|W(\P)| \ll_\W |I|^{d_\W}$$ 
for every $\P \in S^\om$. Taking $\log \Delta$ and applying the expression above gives
\begin{equation} \label{upbnd}
\frac{\log \Delta}{|S|^\om} =\frac{\sum_{\P}^* \log |W(\P)|}{|S|^\om} \le d_{\W} \log |I| + O_\W(1).
\end{equation}
Fix any irreducible polynomial $f$ with $|f| \le N$, where $N$ is to be set at the end. Then for every point $P \in (\F_q[T])^2$ let $\r_P$ be the fraction of points in $S$ that reduce to $P \mmod f$. For each $\P$ let $\k(\P) \in \{0,1,\ldots,\om-1\}$ be defined in a way that $\om-\k(\P)$ is the number of distinct points among the points $P_i \mmod f$. Then one can state
\begin{equation}\label{ord}
\ord_f \Delta \ge \sum_\P^*\k(\P) = \sum_\P\k(\P) - \sum_\P^{na}\k(\P),
\end{equation}
where the first sum on the right hand side is taken over all $\P$ and the second one is the sum over all inadmissible ensembles $\P$.

We are going to proceed in two steps. First, we will calculate the sum over all $\P \in S^\om$ by probabilistic methods. Here we see $P_1,\ldots, P_\om$ as $\om$ independent random variables with values in $(\F_q[T])^2$ and use
$$Y_{P} = \begin{cases} 1, & \mbox{if at least one of } P_i \in S/\{P\} \mbox{ is equal to } P \mmod f;\\ 0, & \mbox{otherwise. } \end{cases}$$
In the inadmissible case of $\P$ we have either at least two points $P_i=P_j$ among the entries of $\P$ or at least two points $P_i = P_j \mmod f$, $P_i, P_j \in \P$, $P_i \neq P_j$. The number of pairs $P_i,P_j$ that satisfy the first possibility can be easily bounded by $O(|S|^{\om-1})$ and for the latter case we permute the entries of our matrix in order to have
$$\det (F_i(P_j))_{1\le i,j\le l} \neq 0$$
of a maximal possible size $l$ and then apply the fact that any $\W$-curve contains at most constant number of elements of $S$.

Let us start with the sum over all $\P \in S^\om$. Consider $\P$ as a random variable with uniform distribution. Then the expected value of the number of distinct points among the $P_i \mmod f$ is equal to
$$\frac{\sum_\P(\om-\k(\P))}{|S|^\om}=\E\left(\sum_P Y_P\right).$$
Further,
\begin{equation*}
\begin{split}
& \E\left(\sum_P Y_P\right)=\sum_P \E(Y_P)=\sum_P \prob(\exists P_i | P_i \equiv P \mmod f) = \sum_P (1-\prob(\not\exists P_i | P_i \equiv P \mmod f))\\
& = \sum_P (1-\prob(\forall P_i | P_i \not\equiv P \mmod f))=\sum_P \left(1-\prod_i \prob(P_i \not\equiv P \mmod f) \right) = \sum_P \left(1-\prod_i (1-\r_P) \right)\\
& = \sum_P  \left(1-(1-\r_P)^\om \right).\\
\end{split}
\end{equation*}
We then have
$$\frac{\sum_\P(\om-\k(\P))}{|S|^\om} = \sum_P \left(1-(1-\r_P)^\om \right).$$
Next
$$\frac{\sum_\P \k(\P)}{|S|^\om} = \frac{\sum_\P \om}{|S|^\om}-\sum_P(1-(1-\r_P)^\om) = \sum_P((1-\r_P)^\om+\om \r_P-1).$$
Since
$$(1-\r_P)^\om+\om \r_P-1 = 1-\om \r_P + \left(\bfrac{\om}{2}\right) \r_P^2 + \ldots + (-1)^\om \left(\bfrac{\om}{\om}\right) \r_P^\om + \om \r_P-1 = \r_P^2\left(\left(\bfrac{\om}{2}\right) - o_{C, \om}(1)\right),$$
then
\begin{equation}\label{sumovall}
\frac{\sum_\P \k(\P)}{|S|^\om} = \frac {\om(\om -1)}{2}\sum_P \r_P^2 - o_{C,\om} \left( \sum_P \r_P^2 \right).
\end{equation}

Now let us bound the sum over all inadmissible $\P$. Consider the set of such $\P$ with $\k(\P) > 0$. Then one of the followings is true:
\begin{enumerate}
\item{There exist $i$ and $j$, such that $P_i=P_j$;}
\item{There exist $i$ and $j$, such that $P_i \equiv P_j\mmod f$, but $P_i \neq P_j$.}
\end{enumerate}
The total number of inadmissible $\P$, such that the first condition above holds is equal to $O(|S|^{\om-1})$.
Let us estimate this number for the second case. Permute the entries in such a way that $i=1$, $j=2$ and $F_1={\bf 1}$, $F_2(P_i) \neq F_2(P_j)$ (this is possible since we have assumed that the elements of $\W$ separate points and $\W$ contains ${\bf 1}$). Then for $l=2$
$$\det (F_i(P_j))_{1 \le i,j \le l} \neq 0.$$
Choose the maximal $l$, such that the above statement still holds.
Then $P_{l+1}$ lies on a $\W$ curve determined by $P_1,P_2,\ldots,P_l$. As we demanded, the number of possible values for $P_{l+1}$ is bounded above by a constant. Then the number of inadmissible $\P$, such that the second case takes place is equal to
$$O_\om(|S|^{\om-3} \delta),$$
where $\delta$ is the number of pairs $(Q_1,Q_2) \in S^2$ that reduce to the same point $\mmod f$. 
By the definition of $\r_P$ we have
$$\delta=|S|^2 \sum_P \r_P^2.$$
Summing two results we see that there are at most
\begin{equation} \label{inadm}
O_\om \left(|S|^{\om-1}+|S|^{\om-3} \delta\right) = 
O_\om \left(|S|^{\om-1} \left(1+\sum_P \r_P^2 \right)\right)= |S|^\om O_\om \left(|S|^{-1} \left(1+\sum_P \r_P^2 \right)\right)
\end{equation}
inadmissible $\P$ with $\k(\P)>0$.
Putting (\ref{sumovall}) and (\ref{inadm}) into (\ref{ord}) we have
\begin{equation*}
\begin{split}
&\frac{\ord_f \Delta}{|S|^\om} \ge \frac{\sum_\P \k(\P) - \sum_{\P}^{na} \k(\P) }{|S|^\om} \ge \left(\frac {\om(\om -1)}{2} - o_{C,\om}(1)\right)\sum_P \r_P^2 - O_\om \left( |S|^{-1} \left(1+\sum_P \r_P^2\right)\right).
\end{split}
\end{equation*}
Using Cauchy's inequality
$$\sum_P \r_P^2 \ge \frac{1}{\alpha |f|} \left(\sum_P \r_P \right)^2 = \frac{1}{\alpha |f|}$$
one can state
$$\frac{\ord_f \Delta}{|S|^\om} \ge \left(\frac{\om(\om-1)}{2} - o_{C, \om}(1)\right) \frac{1}{\alpha |f|} - O_{\om,\alpha, |f|} \left( |S|^{-1}\right).$$
Multiply the equation above by $\log |f|$ and sum over all $|f| \le N$: 
\begin{equation}\label{pred}
\sum_{|f| \le N} \log |f| \left(\frac{\om (\om-1)}{2} - o_{C, \om}(1)\right) \frac{1}{\alpha |f|} + O_{\om,\alpha} \left( |S|^{-1} \sum_{|f| \le N}\log |f|\right) \le \frac{\log \Delta}{|S|^\om}.
\end{equation}
As we know from (\ref{upbnd})
$$\frac{\log \Delta}{|S|^\om} \le d_\W \log |I| + O_\W(1).$$
Applying this estimate to (\ref{pred}) gives
$$\frac{\om(\om-1)}{2\alpha} \sum_{|f| \le N} \frac{\log |f|}{|f|} + O_{\om,\alpha} \left( |S|^{-1} \sum_{|f| \le N}\log |f|\right) -o_{C,\om,\alpha} \left(\sum_{|f| \le N}\frac{\log |f|}{|f|}\right)\le d_\W \log |I| + O_\W(1).$$
Taking $N = |S|$ we end with
$$|S| \ll_{\om,\W} |I|^{\frac{2\alpha d_\W}{\om(\om-1)} + o_{\alpha,C}(1)}.$$
\end{proof}

\begin{lemma}\label{lintr}
Let $\Cc$ be an irreducible algebraic curve of degree $d$ over $\F_q[T]$ which is defined by $F(X,Y)=0$. There exists a linear transformation
$$(X,Y) \to (X',Y')$$
such that $\deg_{X'} F(X',Y')=d$.
\end{lemma}
\begin{proof}
We can assume $\deg_X F(X,Y) <d$, otherwise we are done.
Any polynomial of the form $F(X,Y) \in (\F_q[T])[X,Y]$ can be written as
$$F(X,Y) = \sum_{\bfrac{i\in J_1}{j \in J_2}} F_{ij}X^i Y^j,$$
where $J_1,J_2 \subset \{0,1,...,d\}$, $F_{ij} \in \F_q$ and $$\max_{\bfrac{i\in J_1}{j \in J_2}}(i+j)=\deg F = d, \;\; \max_{i \in J_1}i = \deg_X F<d.$$
Consider a linear transformation
$$(X,Y) \to (X',Y')$$
such that $(X,Y)=(AX'+BY',CX'+DY')$, where $A,B,C,D \in \F_q[T]$ with $AD-BC \neq {\bf 0}$.
Changing the variables $(X,Y) \to (X',Y')$
we obtain
\begin{equation*}
\begin{split}
& F(X,Y) = \sum_{\bfrac{i\in J_1}{j \in J_2}} F_{ij}(AX'+BY')^i (CX'+DY')^j \\
&=\sum_{\bfrac{i\in J_1}{j \in J_2}}  \sum_{k=0}^i \sum_{l=0}^j \left(\bfrac{i}{k}\right)\left(\bfrac{j}{l}\right) F_{ij} A^{i-k}B^{k}C^{j-l}D^{l} (X')^{i+j-k-l}(Y')^{k+l}.\\
\end{split}
\end{equation*}
In new variables $(X',Y')$ we have
$$\deg_{X'}F=\max_{\bfrac{k\in \{0,\ldots,i\},i\in J_1}{l\in \{0,\ldots,j\},l\in J_2}}(i+j-k-l),$$
which is equal to $d$, since $\max_{\bfrac{i\in J_1}{j \in J_2}}(i+j)=\deg F = d$.
\end{proof}

\section{Proof of the theorem}
We start with an interpolation argument, which is used for a similar goal in \cite{h-b}.
Let again $F \in (\F_q[T])[X,Y]$ be written in a form
$$F(X,Y) = \sum_{\bfrac{i\in J_1}{j \in J_2}} F_{ij}X^i Y^j,$$
where $J_1,J_2 \subset \{0,1,...,d\}$, $F_{ij} \in \F_q$.
We are counting the number of distinct lattice points $P=(X,Y) \in I^2 \cap \Cc$. If we have less than $r(d)=d^2+1$ such points, then we are done. Suppose that we have at least $r(d)$ points: $P_i=(X_i,Y_i) \in C\cap I^21,$ $i =1,\dots,r(d)$ with $F(P_i)={\bf 0}$. Denote by $n(d)=\frac{1}{2} (d+1)(d+2)$ the number of monomials of degree less or equal than $d$. Consider $n(d) \times r(d)$ matrix $A$, whose $i$-th row consists of the monomials of degree $d$ in the variables $X_i,Y_i$. Let $\vec{b} \in F_q^{n(d)}$ be a vector, whose entries are the corresponding coefficients $F_{ij}$ of $F(X,Y)$. For such a vector $\vec{b}$ we have an equation
$$A\vec{b}=\vec{0}.$$
Since $\vec{b}\neq \vec{0}$, then the matrix $A$ has a rank less than or equal to $n(d)-1$.
Thus there is a solution $\vec{g}\neq \vec{0}$, where $\vec{g}$ is constructed out of the minors of $A$ with $|\vec{g}| \ll_d |I|^{dn(d)}$. Let $G \in (\F_q[T])[X,Y]$ be the form of degree $d$ corresponding to the vector $\vec{g}$. Then $G(X,Y)$ and $F(X,Y)$ share $r(d)$ zeros (points $P_i$).  By B\'{e}zout's theorem it is possible only if $G$ is a multiple of $F$. Since $F$ is irreducible, then $G$ is also irreducible and defines the same curve $\Cc$. Let us work with $G$ instead of $F$.

We are going to proceed in two steps:
\begin{enumerate}
\item{If $\deg_X G < d$, then by Lemma \ref{lintr} we can change variables so that $\deg_{X'} G = d$. If not, then proceed to the next step.}
\item{Using Weil bounds we obtain
$$|\{(X,Y) \in (\F_q[T] \mmod f)^2:\; G(X,Y)=0 \mmod f\}| = |f|+O_d(\sqrt{|f|}).$$
Further, for every $\veps > 0$ and for every irreducible polynomial $f\in \F_q[T]$ with the condition $|f|\ge c(\veps)$ the set $S$ intersects at most $\left(1+\frac{\veps}{2}\right) |f|$ residue classes $\mmod f$ (here $c(\veps)$ is a constant that depends only on $\veps$).
Applying Lemma \ref{helf-ven} with $\alpha=1+\frac{\veps}{2}$ and $\W$ from Example \ref{wset}: $\W=\W_{d-1,M}$ we obtain
$$|S|\ll_{\veps,\W} |I|^{\frac{\left(1+\frac{\veps}{2}\right)(d+M-1)}{(d(M+1)-1))}+o_{\veps, C}(1)}.$$
We choose $M$ to be large enough and end with
$$|S|\ll_{\veps,\W} |I|^{\frac{1}{d}+\frac{3\veps}{4}+o_{\veps,C}(1)}.$$
}
\end{enumerate}

\section{An application to counting elliptic curves}
In this section we are going to proceed with counting the number of elliptic curves $E_{a,b}$ with coefficients $a,b $ living in a small box that lie in the same isomorphic classes. This is basically the generalization of several statements presented in \cite{shparlinski2}. Doing this we have an opportunity to apply Theorem \ref{thm} and also to show that some results for number fields can be also adapted to function fields. 

Let $I$ stand again for an interval of polynomials of the form $X(T)+Y(T)$, where $X(T) \in \F_q[T]$ is a fixed polynomial
and $Y(T) \in \F_q[T]$ runs through all polynomials of degree less or equal than $d$. The coefficients of $X$ and $Y$ belong to $\F_q$ just as in section 2. 

For a prime power $q$ we consider a family of elliptic curves $E_{a,b}$
$$E_{a,b}: Y^2=X^3+aX+b,$$
where $X$ and $Y$ belong to $\F_q[T]$ as before and $a$, $b$ are some coefficients from $\F_q[T]$ with the property that $4a^3+27b^2 \neq {\bf 0}$.
As in the number field case we say that two curves $E_{a,b}$ and $E_{c,d}$ are isomorphic if
$$at^4 \equiv c (\mmod f) \;\;\text{ and }\;\; bt^6 \equiv d \;(\mmod f).$$
The existence of an isomorphism between $E_{a,b}$ and $E_{c,d}$ implies that
\begin{equation}
\label{32} a^3 d^2 \equiv c^3 b^2 \;(\mmod f)
\end{equation}
for some $f\in \F_q[T]$.
We denote by $N(I^2)$ the number of solutions to (\ref{32}) with $(a,b),(c,d) \in I^2$.
Then for $\lambda \in \F_q[T]$ we write $N_{\lambda}(I^2)$ for the number of solutions to the congruence
$$a^3 \equiv \lambda b^2 \;(\mmod f), \;\; (a,b) \in I^2.$$
We are going to give an upper bound on $N_{\lambda}(I^2)$ that implies upper bounds for the number of elliptic curves $E_{a,b}$ with coefficients $a,b \in I$ that lie in the same isomorphic classes.

For a polynomial $X\in \F_q[T]$ and an irreducible polynomial $f\in \F_q[T]$ we use $\{X\}_f$ to denote
$$\{X\}_f=\min_{Y \in \F_q[T]} |X-fY|=\min_{Y \in \F_q[T]} q^{\deg (X-fY)}.$$

From Dirichlet pigeon-hole principle we obtain
\begin{lemma} \label{pigeon}
For real numbers $T_1,\ldots, T_s$ with $1 \le T_1,\ldots,T_s \le |f|$, $T_1 \cdots T_s \ge |f|^{s-1}$ and any polynomials $X_1,\ldots,X_s \in \F_q[T]$ there exists a polynomial $t\in \F_q[T]$ such that $t$ is not a multiple of $f$ and
$$\{X_i t\}_f\ll T_i,\;\;i=1,\ldots,s.$$
\end{lemma}

Now we can give a good bound for $N_\lambda(I^2)$:
\begin{theorem} \label{1/9}
Let $I$ be an interval of polynomials of degree less or equal than $d$ with coefficients in $\F_q$ and the length of $I$ is $|I|=q^d$.
For any irreducible polynomial $f \in \F_q[T]$ such that $1 \le |I| \le |f|^{\frac{1}{9}}$ and for any $\lambda \in \F_q[T]$ we have
$$N_\lambda(I^2) \le |I|^{\frac{1}{3}+o(1)}.$$
\end{theorem}
\begin{proof}
We have to estimate the number of solutions to 
$$(X+X_0)^3 \equiv \lambda (X_0+Y)^2 \;(\mmod f).$$
This congruence is equivalent to
\begin{equation} \label{congr}
X^3 + 3XX_0^2 +3X^2X_0 - \lambda Y^2 - 2\lambda X_0Y \equiv \lambda X_0^2 - X_0^3 \; (\mmod f).
\end{equation}
For any $T \le q^{\frac{1}{4}} / |I|^{\frac{1}{2}}$ we can apply Lemma \ref{pigeon} to 
$$X_1=1,\;X_2=3X_0,\;X_3=3X_0^2,\;X_4=-\lambda,\;X_5=-2\lambda X_0$$
and
$$T_1=T^4 |I|^2,\;T_2=T_4=\frac{|f|}{T|I|}, T_3=T_5=\frac{|f|}{T}$$
and find that there exists $t$ with $|t|\le T^4 |I|^2$ such that

$$\{3X_0 t\}_{f} \le \frac{|f|}{T|I|},\; \{3X_0^2 t\}_{f} \le \frac{|f|}{T},\;\{\lambda t\}_{f} \le \frac{q}{T|I|},\; \{2\lambda X_0 t\}_{f} \le \frac{|f|}{T}.$$
For $i=1,\ldots,5$ denote by $f_i$ a polynomial which satisfies $f_i=X_i t$.
Then multiply (\ref{congr}) by $t$ leads us to the equality
\begin{equation} \label{congrf}
f_1 X^3 + f_2 X^2 +f_3 X + f_4 Y^2 +f_5 Y + f_6 = |f| Z,
\end{equation}
where 
$$|f_1| \le T^4 |I|^2,\;|f_2|,|f_4| \le \frac{|f|}{T|I|},\;|f_3|,|f_5| \le \frac{|f|}{T}, \;|f_6| \le \frac{|f|}{2}.$$
Since for $X,Y \in I$ we have $|X|,|Y| \le |I|$, then the left hand side of (\ref{congrf}) is bounded above by $T^4 |I|^5 + \frac{4|f||I|}{T} + \frac{|f|}{2}$. Thus 
$$|Z| \ll \frac{T^4 |I|^5}{|f|} + \frac{4|I|}{T} +1.$$
Choosing $T \approx \frac{|f|^{\frac{1}{5}}}{|I|^{\frac{4}{5}}}$ and applying the condition $1 \le |I| \le |f|^{\frac{1}{9}}$ we end with the bound
$$|Z| \ll \frac{|I|^{\frac{9}{5}}}{q^{\frac{1}{5}}}+1 \ll 1.$$

\end{proof}

Application of Theorem \ref{1/9} to the family of curves $E_{x^2,x^3}$ with $|x| \le |I|^{\frac{1}{3}}$ shows that the result of Theorem \ref{1/9} can not be improved.
Thus in general we are not able to get any bound stronger than $N_\lambda (I^2) = O(|I|^{\frac{1}{3}})$.


\end{document}